\documentclass{amsart}

\usepackage{xparse}                                             %%% Define some macros

\usepackage{amsmath}
\usepackage{amsthm}

\usepackage[style=numeric,giveninits=true,maxalphanames=99,maxbibnames=99]{biblatex}
\DeclareDatamodelConstant[type=list]{nameparts}{family,given,given-i}
\bibliography{Tight block designs.bib}

\numberwithin{equation}{section}
\numberwithin{figure}{section}
\numberwithin{table}{section}

\usepackage{amsfonts}
\usepackage{amssymb}
\usepackage{mathtools}
\makeatletter
\DeclareRobustCommand*\cal{\@fontswitch\relax\mathcal}
\makeatother

\usepackage{graphicx}

\usepackage[colorlinks]{hyperref}
\usepackage[capitalize]{cleveref}
\crefname{section}{Section}{Sections}
\Crefname{section}{Section}{Sections}
\Crefname{construction}{Construction}{Constructions}
\Crefname{construction}{Construction}{Constructions}
\usepackage{url}
\newcommand{\doi}[1]{\textsc{doi}: \href{http://dx.doi.org/#1}{\nolinkurl{#1}}}

\usepackage[margin = 1in]{geometry}

\usepackage{enumitem}
\setlist[enumerate, 1]{label = (\roman*), font = \upshape}
\setlist[itemize, 2]{label = {$\circ$}}
\usepackage{longtable}
\usepackage{subcaption}

\theoremstyle{plain}
\newtheorem{theorem}{Theorem}[section]
\newtheorem{corollary}[theorem]{Corollary}
\newtheorem{lemma}[theorem]{Lemma}

\newtheorem{proposition}[theorem]{Proposition}

\theoremstyle{definition}

\theoremstyle{remark}
\newtheorem{remark}[theorem]{Remark}
\newtheorem*{acknowledgment}{Acknowledgment}
\newtheorem*{notation}{Notation}

\DeclareMathOperator{\Z}{\mathbb{Z}}                            %%% Integer Ring
\DeclareMathOperator{\Q}{\mathbb{Q}}                            %%% Rational Field
\DeclareMathOperator{\ZVal}{\mathbb{Z}\mathrm{Val}}             %%% Integer Valued
\DeclareMathOperator{\RePart}{\mathrm{Re}}                      %%% Integer Valued

\DeclareMathOperator{\gt}{>}                                    %%% Greater Than
\DeclareMathOperator{\lt}{<}                                    %%% Less Than

\renewcommand{\mid}{:\ }                                        %%% Such That (Set Construction)

\def\pFq#1#2#3#4#5{\begingroup%
    {}_{#1}F_{#2}\biggl[\begin{matrix}#3\\#4\end{matrix}\biggm.;#5\biggr]
	\endgroup}

\begin{document}

\title{Classification of tight $2s$-designs with $s \geq 2$}

\author{Ziqing Xiang}
\address{Department of Mathematics and National Center For Applied Mathematics Shenzhen, Southern University of Science and Technology}
\email{xiangzq@sustech.edu.cn}

\date{}

\begin{abstract}
Tight $2 s$-designs are the $2 s$-$(v, k, \lambda)$ designs whose sizes achieve the Fisher type lower bound ${v \choose s}$. Symmetric $2$-designs, the Witt $4$-$(23, 7, 1)$ design and the Witt $4$-$(23, 16, 52)$ design are tight designs. It has been widely conjectured since 1970s that there are no other nontrivial tight designs. In this paper, we give a proof of this conjecture. In the proof, an upper bound $v \ll s$ is shown by analyzing the parameters of the designs and the coefficients of the Wilson polynomials, and a lower bound $v \gg s (\ln s)^2$ is shown by using estimates on prime gaps.
\end{abstract}

\maketitle

\section{Introduction}
\label{sec:subexternal}

For positive integers $t$, $v$, $k$ and $\lambda$, a {\em $t$-$(v, k, \lambda)$ design}, or simply a {\em $t$-design}, consists of a $v$-set $V$, whose elements are called {\em points}, and a collection $B$ of $k$-subsets of $V$, whose elements are called {\em blocks}. The defining property of a $t$-design is that every $t$-subset of $V$ is contained in exactly $\lambda$ blocks. A design is {\em nontrivial} if $B$ is not the whole collection of all $k$-subsets of $V$.

In 1977, Wilson \cite{Wilson1972a,Wilson1972,Wilson1975} constructed $2$-designs for all sufficiently large $v$ under some necessary divisibility conditions. In 1987, designs for all $t$ and some $v, k, \lambda$ were first constructed by Teirlinck \cite{Teirlinck1987}. In 2014, Keevash \cite{Keevash2014} developed the method of randomized algebraic construction, and in particular, constructed $t$-designs for all sufficiently large $v$ under the divisibility conditions.

The focus of this paper is on the sizes of designs. Fisher \cite{Fisher1940} proved in 1940 that, for nontrivial $2$-designs, the number of blocks is at least the number of points, namely, $|B| \geq v$. This inequality is known as the {\em Fisher's inequality}. In 1968, Petrenjuk \cite{Petrenyuk1968} generalized Fisher's inequality and proved that for nontrivial $4$-designs, $|B| \geq {v \choose 2}$. Moreover, Petrenjuk proposed a conjectural inequality $|B| \geq {v \choose s}$ for $2s$-designs. In 1975, Ray-Chaudhuri and Wilson \cite{RayChaudhuriWilson1975} proved Petrenjuk's conjecture. This lower bound is known as the {\em Fisher type lower bound}. Designs that achieve this lower bound are called {\em tight}.

Tight $2$-designs are also called {\em symmetric designs}. Typical examples are the $2$-$(n^2 + n + 1, n + 1, 1)$ designs arising from projective planes of order $n$, and the $2$-$(4 n - 1, 2 n - 1, n - 1)$ designs arising from Hadamard matrices of order $4 n$. For tight $4$-designs, there are the Witt $4$-$(23, 7, 1)$ design and its complement, the Witt $4$-$(23, 16, 52)$ design, are tight $4$-designs. No other nontrivial tight $2s$-designs with $s \geq 2$ have been found.

Due to a series work \cite{Delsarte1973,RayChaudhuriWilson1975,Peterson1977,Bannai1977,EnomotoItoNoda1979,Bremner1979} in 1970s by Delsarte, Ray-Chaudhuri, Wilson, Bannai, Peterson, Enomoto, Ito, Noda and Bremner, it was known that (1) the only nontrivial tight $4$-designs are the Witt designs; (2) there are no nontrivial tight $6$-designs; (3) there are only finitely many nontrivial tight $2s$-designs for each $s \geq 5$. Therefore, it was widely conjectured that symmetric designs and the Witt designs are the only nontrivial tight designs.

There was no further progress on this conjecture until recently. In 2013, Dukes and Short-Gershman \cite{DukesShortGershman2013} proved the conjecture for $s \in \{5, 6, 7, 8, 9\}$ by extending Bannai's approach \cite{Bannai1977}. In 2018, the author \cite{Xiang2018} used a different approach to prove the $s = 4$ case.

In this paper, we have completely proved the conjecture by using new strategies.

\begin{theorem}[Main result] \label{thm:myxocyte}
The only nontrivial tight $2 s$-designs with $s \geq 2$ are the Witt $4$-$(23, 7, 1)$ design and the Witt $4$-$(23, 16, 52)$ design.
\end{theorem}

In coding theory, the perfect $e$-code is the dual concept to the tight $2s$-design. Tiet{\"a}v{\"a}inen \cite{Tietaevaeinen1972} proved in 1972 that the only nontrivial perfect $e$-codes with $e \geq 2$ on Hamming association schemes over finite fields are the $[23, 12, 7]_2$ and $[11, 6, 5]_3$ Golay codes. The classification of perfect $e$-codes on Johnson association schemes, where $t$-$(v, k, \lambda)$ designs reside, is still open. \cref{thm:myxocyte} can be viewed as the design theory counterpart of the classfication of such perfect codes.

\subsection{Intersection numbers and the Wilson polynomials}

An important tool in the study of block designs is the {\em intersection numbers} of the designs. These are the sizes of the intersections of different blocks in a design.

For a tight $2 s$-$(v, k, \lambda)$ design, its intersection numbers are completely determined by $s$, $v$ and $k$. More precisely, combining the results of Delsarte \cite{Delsarte1973} and those of Ray-Chaudhuri and Wilson \cite{RayChaudhuriWilson1975}, for a fixed $s$, there exists a degree $s$ polynomial $\Phi_s(v, k; z) \in \Q(v, k)[z]$ in $z$ with coefficients in $\Q(v, k)$, such that the zeros of $\Phi_s(v, k; z)$ are exactly the intersection numbers of a nontrivial tight $2 s$-$(v, k, \lambda)$ design. These polynomials $\Phi_s$ are called the {\em Wilson polynomials}. They are the design theory counterpart of the Lloyd polynomials for perfect codes.

Since the intersection numbers are all integers, so are the zeros of the Wilson polynomials. This gives a necessary condition for the existence of nontrivial tight designs. Most previous work has focused on the behavior of the zeros of the Wilson polynomials.

When $s = 2$ or $s = 3$, the zeros of $\Phi_s$ can be computed. In 1979, Enomoto-Ito-Noda \cite{EnomotoItoNoda1979} analyzed the zeros of $\Phi_2$, and proved that the nontrivial tight $4$-designs give integer points on some explicit elliptic curve. Soon after this result, Bremner \cite{Bremner1979} and Stroeker \cite{Stroeker1981} independently determined all integer points on this elliptic curve, thus proving \cref{thm:myxocyte} for $s = 2$. In 1977, Peterson \cite{Peterson1977} analyzed the zeros of $\Phi_3$, and proved \cref{thm:myxocyte} for $s = 3$.

When $s \geq 5$, in 1977, Bannai \cite{Bannai1977} obtained the asymptotic behavior of the zeros of $\Phi_s$ as $v \to \infty$, and showed that they are controlled by the zeros of the Hermite polynomials. Using the properties of the Hermite polynomials, he proved that for each fixed $s \geq 5$, there are only finitely many nontrivial tight $2 s$-designs. In 2013, Dukes and Short-Gershman \cite{DukesShortGershman2013} applied Bannai's approach explicitly and proved \cref{thm:myxocyte} for $s \in \{5, 6, 7, 8, 9\}$.

Unfortunately, Bannai's method does not seem to be able to prove \cref{thm:myxocyte}. It does not apply for the $s = 4$ case since the zeros of the degree $4$ Hermite polynomial do not satisfy the needed properties. Moreover, for large $s$, some heuristics by the author suggest that an explicit version of Bannai's approach might only give a loose upper bound $v \ll \exp(s)$ as $s \to \infty$. \cref{rem:osteomere} elaborates on why this upper bound is not small enough to prove \cref{thm:myxocyte}.

In 2018, the author \cite{Xiang2018} studied the coefficients of the Wilson polynomials, instead of the zeros of the Wilson polynomials, and thus obtained some different necessary conditions for the existence of nontrivial tight designs. By combining these conditions and results on $\Phi_4$ in \cite{Bannai1977,DukesShortGershman2013}, the author \cite{Xiang2018} proved \cref{thm:myxocyte} for $s = 4$. However, the results we need on $\Phi_4$ are unique to the $s = 4$ case, and new ideas are needed to prove \cref{thm:myxocyte}.

\subsection{Strategy and organization}

The proof of the main result consists of two parts: showing a good upper bound for $v$, and showing a good lower bound for $v$. In both parts, we use some new strategies.

The main idea of proving an upper bound is to construct an auxiliary function $H$ in $v$ and $k$ for each fixed $s$ such that the following three properties hold.
\begin{enumerate}
	\item \label{itm:synema} The function $H$ takes rational values with bounded denominators at the parameters of nontrivial tight $2 s$-designs.
	\item \label{itm:Cerdonian} As $v, k \to \infty$ with $v \geq k$, $H = o(1)$.
	\item \label{itm:pectoriloquism} $H \gt 0$.
\end{enumerate}
Since rational numbers with fixed bounded denominators are discrete, \ref{itm:synema}, \ref{itm:Cerdonian} and \ref{itm:pectoriloquism} conflict with each other when $v$ is sufficiently large compared with $s$, which leads to an upper bound for $v$ in terms of $s$.

Such an $H$ is constructed in \cref{sec:biliously,sec:carouser}. There are two families of basic functions $\lambda_i$ and $\alpha_i$ with $\lambda_i = O(v^s)$ and $\alpha_i = O(v^i)$. We first construct certain rational linear combinations $h_i = O(v)$ of $\lambda_i$'s and $\alpha_i$'s in \cref{sec:biliously}. We then construct a certain rational quadratic combination $H$ of $h_i$'s in \cref{sec:carouser}. This $H$ satisfies \ref{itm:synema} because $\lambda_i$'s and $\alpha_i$'s do. Moreover, this $H$ happens to have a simple closed-form formula, which allows us to prove $H = O(v^{- \lfloor s / 2 \rfloor + 2})$. Therefore, \ref{itm:Cerdonian} holds for $s \geq 6$. The property \ref{itm:pectoriloquism} also follows from the closed-form formula of $H$.

In \cref{sec:ineffaceability}, we give an upper bound for the denominators of $H(v_0, k_0)$ for parameters $(v_0, k_0)$ of nontrivial tight $2 s$-designs. This gives an upper bound $v \ll s$ as $s \to \infty$. The result in \cref{sec:ineffaceability} relaxes the bounded denominator condition in \ref{itm:synema} a little bit, since otherwise we could only prove $v \ll s^5$.

On the other hand, we analyze the parameter $\lambda$ to get a lower bound for $v$. The denominator of $s! \lambda$ is a product of consecutive integers, which cannot contain any primes. Using known estimates of prime gaps, we prove $v \gg s (\ln s)^2$ as $s \to \infty$ in \cref{sec:gaffer}.

These two bounds prove \cref{thm:myxocyte} for sufficiently large $s$. To give a proof for all $s \geq 2$, we need to get explicit versions of the upper bound $v \ll s$ and the lower bound $v \gg s (\ln s)^2$, which is done in \cref{sec:ineffaceability} and \cref{sec:gaffer}, respectively. \cref{sec:malandrous} handles the small $s$ cases by giving explicit upper and lower bounds for $v$. It also contains the full proof of \cref{thm:myxocyte}.

\begin{notation}
    The set of all integers in a real interval $[a, b]$ is denoted by $[a, b]_{\Z}$. The {\em $n$-th rising factorial} and the {\em $n$-th falling factorial} of $x$ are denoted by $x^{\overline{n}} := \prod_{i = 0}^{n - 1}(x + i)$ and $x^{\underline{n}} := \prod_{i = 0}^{n - 1}(x - i)$, respectively. We adopt the convention that $x^{\overline{0}} = x^{\underline{0}} = x^0 = 1$. The {\em $p
    $-adic valuation} of a rational number $x$ is denoted by ${\rm val}_p(x)$.
\end{notation}

\begin{acknowledgment}
I would like to thank Pei Yu for inspiring me to find the proof of \cref{thm:myxocyte}. I also thank Eiichi Bannai, Akihiro Munemasa, Daniel Nakano, Chun-Ju Lai, Hajime Tanaka for useful discussions.
\end{acknowledgment}

\section{Bounded denominators of $h_{s, i}$'s}
\label{sec:biliously}

This section studies several families of functions in $v$ and $k$ that takes rational values with bounded denominators at the parameters of nontrivial tight designs.

Let $\ZVal[v, k]$ be the ring of integer-valued rational functions in $v$ and $k$, ${\cal D}_s$ the collection of all $(v_0, k_0)$'s such that there exists a nontrivial tight $2 s$-$(v_0, k_0, \lambda_0)$ design, and ${\cal A}_s$ the ring of integer-valued functions on ${\cal D}_s$. Then, $\Q \otimes_{\Z} {\cal A}_s$ is the ring of rational-valued functions on ${\cal D}_s$ with bounded denominators. It is clear that $\ZVal[v, k] \subseteq {\cal A}_s \subseteq \Q \otimes_{\Z} {\cal A}_s$.

For designs, there are three families of important functions. For each $i \in [0, s]_{\Z}$, let
\begin{align}
\label{eq:Bee}
\lambda_{s, i} := & \frac{1}{s!} \frac{k^{\underline{s + i}}}{(v - s)^{\underline{i}}} \in \Q(v, k), \\
\label{eq:Bass}
\alpha_{s, i} := & {s \choose i} \frac{(k - s)^{\overline{i}}(k - s + 1)^{\overline{i}}}{(v - 2 s + 1)^{\overline{i}}} \in \Q(v, k), \\
\label{eq:Sycamore}
h_{s, i} := & \frac{(k - s)^{\overline{i + 1}}}{(v - 2 s + 1)^{\overline{i}}} \in \Q(v, k).
\end{align}

It is well known that $\lambda_{s, i} \in {\cal A}_s$, since every tight $2 s$-$(v, k, \lambda)$ design is an $(s + i)$-$(v, k, \lambda_{s, i})$ design. It was proved in \cite[Corollary 2.2]{Xiang2018} that $\alpha_{s, i} \in {\cal A}_s$. More precisely, the zeros of the Wilson polynomial in $z$,
\[
	\Phi_s(z) := \sum_{i = 0}^s (-1)^{s - i} \frac{{v - s \choose i} {k - i \choose s - i} {k - i - 1 \choose s - i}}{{s \choose i}} {z \choose i} = \frac{{v - s \choose s}}{s!} \left(\sum_{i = 0}^s (-1)^i \alpha_{s, i} z^{\underline{s - i}}\right),
\]
are exactly the $s$ intersection numbers of tight $2 s$-designs by \cite{Delsarte1973,RayChaudhuriWilson1975,Bannai1977}. Therefore, $\alpha_{s, i}$ must take integer values at $(v_0, k_0) \in {\cal D}_s$.

The goal of this section is to prove $h_{s, i} \in \Q \otimes_{\Z} {\cal A}_s$. In other words, to prove that $h_{s, i}(v_0, k_0)$ has bounded denominator for $(v_0, k_0) \in {\cal D}_s$. The author proved such a result in \cite[Lemma 4.1]{Xiang2018} that the denominators of $h_{s, i}(v_0, k_0)$ are bounded above by $\exp(\frac{1}{2} s^2 \ln s + O(s^2))$. The main result of this section, \cref{prop:expostulating}, gives a much better estimate. It shows that the denominators are bounded above by $\exp(s \ln s + O(s))$.

Note that such an improvement is not necessary to prove \cref{thm:myxocyte} for sufficiently large $s$. However, for small $s$, this improvement can reduce the amount of computation needed in \cref{sec:malandrous} by a factor of $10^{50}$. So, it is an essential step in the proof of \cref{thm:myxocyte}.

\subsection{Linear combinations of $\alpha_{s, i}$ and $\lambda_{s, i}$}

For a natural number $n$, let
\[
	\ell_n := {\rm lcm} \left\{ {i \choose j} \mid 0 \leq j \leq i \leq n \right\}.
\]
The number $\ell_n$ has many alternative expressions, for example, $\ell_n = {\rm lcm} \{1, \dots, n\}$. Moreover, the statement $\ell_n = \exp(n + o(n))$ is equivalent to the prime number theorem. This is a technical number we need in the estimates in this section.

We first express $h_{s, s}$ in terms of $\lambda_{s, i}$'s and $\alpha_{s, i}$'s. The quotients $\frac{\lambda_{s, 1}}{h_{s, s}}, \dots, \frac{\lambda_{s, s}}{h_{s, s}}$ and $\frac{\alpha_{s, 1}}{h_{s, s}}, \dots, \frac{\alpha_{s, s}}{h_{s, s}}$ are all rational functions in $v$ and $k$. In \cref{lem:palatefulness}, we express $1$ as an explicit rational linear combination of these rational functions, from which we conclude that $h_{s, s}$ is a rational linear combination of $h_{s, 1}, \dots, h_{s, s}$ and $\alpha_{s, 1}, \dots, \alpha_{s, s}$.

\begin{lemma} \label{lem:palatefulness}
Let $s$ be a natural number with $s \geq 2$. Then, $s! \ell_{s - 2} h_{s, s}$ is a $\ZVal[v, k]$-linear combination of $\alpha_{s, 1}, \dots, \alpha_{s, s}$ and $\lambda_{s, 1}, \dots, \lambda_{s, s}$.
\end{lemma}

\begin{proof}
As polynomials in $\Q[v, k]$, we have the identity
\begin{equation}
\label{eq:logopedia}
\begin{aligned}
1
& =	\sum_{i = 0}^{s - 1} \left( \sum_{j = i + 1}^{s - 1} \frac{ (-1)^{i + j} (j - i - 1)! }{ (s - i - 1)! (j - 1)! }  {k - j - 1 \choose s - j - 1} {k - i - 1 \choose j - i - 1} {v - s \choose i} \right) (v - 2 s + 1)^{\overline{s - i - 1}} (k - 1)^{\underline{i}} \\
& + \sum_{i = 0}^{s - 1} \left( \sum_{j = 1}^{i} \frac{ (-1)^{i + j} (i - j)! }{ i! (s - j - 1)! } {k - j - 1 \choose i - j} {k - 1 \choose j - 1} {v - s - i - 1 \choose s - i - 1} \right) (v - s)^{\underline{i}} (k - s + 1)^{\overline{s - i - 1}}.
\end{aligned}
\end{equation}
since the right side of \cref{eq:logopedia} is a polynomial in $v$ of degree $\leq s - 1$, a polynomial in $k$ of degree $\leq s - 2$, and a direct calculation shows that equality holds for $(v, k) \in \{s, \dots, 2 s - 1\} \times \{1, \dots, s - 1\}$.

By \cref{eq:Bee,eq:Bass,eq:Sycamore},
\begin{align}
	\label{eq:mangerite}
	s! \frac{\lambda_{s, i + 1}}{h_{s, s}} = & (v - 2 s + 1)^{\overline{s - i - 1}} (k - 1)^{\underline{i}}, \\
	\label{eq:unfarrowed}
	\frac{ i! i! (s - i)! }{ s! } {k \choose i} \frac{\alpha_{s, s - i}}{h_{s, s}} = & (v - s)^{\underline{i}} (k - s + 1)^{\overline{s - i - 1}}.
\end{align}
Substituting \cref{eq:unfarrowed,eq:mangerite} into \cref{eq:logopedia} gives
\begin{align*}
s! \ell_{s - 2}
& =	\sum_{i = 0}^{s - 1} \left( \sum_{j = i + 1}^{s - 1} \frac{ (-1)^{i + j} s! s! (j - i - 1)! \ell_{s - 2}}{ (s - i - 1)! (j - 1)! }  {k - j - 1 \choose s - j - 1} {k - i - 1 \choose j - i - 1} {v - s \choose i} \right) \frac{\lambda_{s, i + 1}}{h_{s, s}} \\
& + \sum_{i = 0}^{s - 1} \left( \sum_{j = 1}^{i} (-1)^{i + j} i! (s - i) \frac{\ell_{s - 2}}{ {s - j - 1 \choose i - j} } {k - j - 1 \choose i - j} {k - 1 \choose j - 1} {v - s - i - 1 \choose s - i - 1} {k \choose i} \right) \frac{\alpha_{s, s - i}}{h_{s, s}},
\end{align*}
from which the result follows.
\end{proof}

We then express $h_{s, i}$ in terms of $\alpha_{s, i}$ and $h_{s, s}$. We use the similar idea as in \cref{lem:palatefulness}, except that in this case we consider the quotients $\frac{\alpha_{s, i}}{h_{s, i}}$ and $\frac{h_{s, s}}{h_{s, i}}$.

\begin{proposition} \label{prop:expostulating}
Let $s$ be a natural number with $s \geq 2$. For every $i \in [0, s]_{\Z}$, $s! \ell_{s - 1} \ell_{s - 2} h_{s, i}$ is a $\ZVal[v, k]$-linear combination of $\alpha_{s, 0}, \dots, \alpha_{s, s}$ and $\lambda_{s, 1}, \dots, \lambda_{s, s}$.
\end{proposition}

\begin{proof}
If $i = 0$, then $h_{s, i} = k - s = (k - s) \alpha_{s, 0}$ by \cref{eq:Sycamore,eq:Bass}, hence the result holds trivially. From now on, we assume that $i \geq 1$.

As polynomials in $\Q[k]$, we have the identity
\begin{equation}
\label{eq:cetology}
\begin{aligned}
	1 & = \left( \sum_{j = 0}^{s - i - 1} (-1)^{s - j} \frac{ (s - i - j)! }{ (s - j - 1)! } {k \choose j} {k - j - 1 \choose s - i - j - 1} \right) (k - s + 1)^{\overline{i - 1}} \\
	& + \left( \sum_{j = s - i + 1}^{s - 1} (-1)^{s - j - 1} \frac{ (i + j - s)! }{ j! } {k - j - 1 \choose s - j - 1} {k - s + i - 1 \choose i + j - s - 1} \right) k^{\underline{s - i}},
\end{aligned}
\end{equation}
since the right side of \cref{eq:cetology} is a polynomial in $k$ of degree $\leq s - i$, and a direct calculation shows that equality holds for $k \in \{0, \dots, s - i - 1\} \cup \{s - i + 1, \dots, s - 1\}$.

By \cref{eq:Bass,eq:Sycamore},
\begin{align}
\label{eq:spangly}
	\frac{ i! (s - i)! }{ s! } \frac{\alpha_{s, i}}{h_{s, i}} = & (k - s + 1)^{\overline{i - 1}}, \\
\label{eq:polytonalism}
 (s - i)! {v - s \choose s - i} \frac{h_{s, s}}{h_{s, i}} = & k^{\underline{s - i}}.
\end{align}
Substituting \cref{eq:spangly,eq:polytonalism} into \cref{eq:cetology} gives
\begin{align*}
	s! \ell_{s - 1} \ell_{s - 2}
	= & \left( \sum_{j = 0}^{s - i - 1} (-1)^{s - j} i! (s - i - j) \ell_{s - 2} \frac{ \ell_{s - 1} }{ {s - j - 1 \choose s - i} } {k \choose j} {k - j - 1 \choose s - i - j - 1} \right) \frac{\alpha_{s, i}}{h_{s, i}}, \\
	= & \left( \sum_{j = s - i + 1}^{s - 1} (-1)^{s - j - 1} \frac{\ell_{s - 1}}{{j \choose s - i}} {k - j - 1 \choose s - j - 1} {k - s + i - 1 \choose i + j - s - 1} \right) {v - s \choose s - i} \frac{s! \ell_{s - 2} h_{s, s}}{h_{s, i}}.
\end{align*}
Therefore, $s! \ell_{s - 1} \ell_{s - 2} h_{s, i}$ is a $\ZVal[v, k]$-linear combination of $\alpha_{s, i}$ and $s! \ell_{s - 2} h_{s, s}$. Then, the result follows from \cref{lem:palatefulness}.
\end{proof}

Since $\ZVal[v, k] \subseteq {\cal A}_s$, any $\ZVal[v, k]$-linear combination of elements in $\Q \otimes_{\Z} {\cal A}_s$ is also in $\Q \otimes_{\Z} {\cal A}_s$. Applying this to \cref{prop:expostulating} gives \cref{cor:epichorion}.

\begin{corollary} \label{cor:epichorion}
For every $i \in [0, s]_{\Z}$, $s! \ell_{s - 1} \ell_{s - 2} h_{s, i} \in {\cal A}_s$.
\end{corollary}

\begin{proof}
The result follows from \cref{prop:expostulating} and that $\alpha_{s, 0}, \dots, \alpha_{s, s} \in {\cal A}_s$ and $\lambda_{s, 0}, \dots, \lambda_{s, s} \in {\cal A}_s $.
\end{proof}

\section{Auxiliary function $H_{s, r}$}
\label{sec:carouser}

Let $s$ be a natural number, and $r$ an even natural number with $r \leq s$. We construct an auxiliary function $H_{s, r}$ in $h_{s, i}$'s as follows:
\begin{equation} \label{eq:chafferer}
	H_{s, r} := \sum_{i = 0}^{r} (-1)^i {r \choose i} h_{s, s - r + i} h_{s, s - i}.
\end{equation}

The goal of this section is to prove that this $H_{s, r}$ satisfies conditions \ref{itm:synema}, \ref{itm:Cerdonian} and \ref{itm:pectoriloquism} in \cref{sec:subexternal}. First, this complicated sum $H_{s, r}$ has a simple closed-form formula as to be shown in \cref{thm:administration}, which is proved in \cref{sec:stack}. The conditions \ref{itm:Cerdonian} and \ref{itm:pectoriloquism} follow directly from the explicit closed-form formula. Second, since \cref{cor:epichorion} proves that $h_{s, 0}, \dots, h_{s, s}$ satisfy the condition \ref{itm:synema}, namely taking rational values with bounded denominators, so does their rational quadratic combination $H_{s, r}$. \cref{prop:foreadapt} is an explicit version of this statement, and is proved in \cref{sec:unflagitious}.

The auxiliary function $H_{s, r}$ is discovered in our attempts to give a proof of \cref{thm:myxocyte}. It is quite mysterious, and at this moment, we do not have any combinatorial interpretations of $H_{s, r}$.

\subsection{Simplifying $H_{s, r}$}
\label{sec:stack}

Dixon's identity is an identity on the hypergeometric function
\[
	\pFq{3}{2}{a & b & c}{ & d & e}{z} := \sum_{i = 0}^\infty \frac{a^{\overline{i}} b^{\overline{i}} c^{\overline{i}}}{d^{\overline{i}} e^{\overline{i}}} \frac{z^i}{i!}.
\]

\begin{theorem}[Dixon's identity] \label{thm:shippy}
Let $a, b, c$ be complex variables. In the region $\RePart (1 + a / 2 - b - c) \gt 0$,
\[
	\pFq{3}{2}{a & b & c}{& 1 + a - b & 1 + a - c}{1} = \frac{\Gamma(1 + a / 2) \Gamma(1 + a / 2 - b - c) \Gamma(1 + a - b) \Gamma(1 + a - c)}{\Gamma(1 + a) \Gamma(1 + a - b - c) \Gamma(1 + a / 2 - b) \Gamma(1 + a / 2 - c)},
\]
where $\Gamma$ is the Gamma function, and both sides should be understood as their analytic continuations.
\end{theorem}

When we choose $a := -r$ for some natural number $r$, $a^{\overline{i}} = 0$ for $i \geq r + 1$. In this case, the hypergeometric function reduces to a polynomial in $z$ with coefficients in $\Q(b, c, d, e)$:
\begin{equation} \label{eq:outland}
	\pFq{3}{2}{-r & b & c}{ & d & e}{z} = \sum_{i = 0}^{r} (-1)^i {r \choose i} \frac{b^{\overline{i}} c^{\overline{i}}}{d^{\overline{i}} e^{\overline{i}}} z^i.
\end{equation}
We have the following specialization of the Dixon's identity.

\begin{corollary}
\label{cor:Sunna}
Let $s$ be a natural number, and $r$ an even natural number. Let $v$ and $k$ be real variables. If $v \geq k + s$ and $k \geq r$, then
\[
\pFq{3}{2}{-r & k - r + 1 & - v + s}{ & - k & v - r - s + 1}{1} = \frac{r^{\underline{r / 2}} (v - k - s)^{\overline{r / 2}}}{k^{\underline{r / 2}} (v - s - r + 1)^{\overline{r / 2}} },
\]
where the both sides are regarded as rational functions in $v$ and $k$.
\end{corollary}

We could prove \cref{cor:Sunna} by setting $b := k - r + 1$, $c := - v + s$ and then taking the limit of \cref{thm:shippy} as $a \to - r$. The inequalities $v \geq k + s$ and $k \geq r$ are to make sure that the summands in the hypergeometric series $_3F_2$ have nonzero denominators and $\RePart(1 + a / 2 - b - c) \gt 0$. The details of the proof are omitted here.

It turns out that when we substitute $h_{s, i}$'s into $H_{s, r}$ in \cref{eq:chafferer}, the complicated sum $H_{s, r}$ is just the left side of the Dixon's identity in disguise. This means that we could use Dixon's identity to give a closed-form formula for $H_{s, r}$.

\begin{theorem} \label{thm:administration}
Let $s$ be a natural number, and $r$ an even natural number with $r \leq s$. Then,
\[
	H_{s, r} = \frac{r!}{(r / 2)!} G_{s, r},
\]
where
\begin{equation} \label{eq:quadrantlike}
	G_{s, r} := \frac{ (v - k - s)^{\overline{r / 2}} (k - s)^{\overline{s - r / 2 + 1}} (k - s)^{\overline{s - r + 1}} }{ (v - 2 s + 1)^{\overline{s}} (v - 2 s + 1)^{\overline{s - r / 2}} }.
\end{equation}
\end{theorem}

\begin{proof}
We proceed with the following calculation.
\begin{align*}
H_{s, r} =
& \sum_{i = 0}^{r} (-1)^i {r \choose i} h_{s, s - r + i} h_{s, s - i} & (\text{by \cref{eq:chafferer}}) \\
= & \sum_{i = 0}^{r} (-1)^i {r \choose i} \frac{(k - s)^{\overline{s - r + i + 1}}}{(v - 2 s + 1)^{\overline{s - r + i}}} \frac{(k - s)^{\overline{s - i + 1}}}{(v - 2 s + 1)^{\overline{s - i}}} & (\text{by \cref{eq:Sycamore}}) \\
= & \frac{ (k - s)^{\overline{s + 1}} (k - s)^{\overline{s - r + 1}} }{ (v - 2 s + 1)^{\overline{s}} (v - 2 s + 1)^{\overline{s - r}} } \cdot \sum_{i = 0}^{r} (-1)^i {r \choose i} \frac{(k - r + 1)^{\overline{i}} (v - s)^{\underline{i}}}{k^{\underline{i}} (v - s - r + 1)^{\overline{i}}} \\
= & \frac{ (k - s)^{\overline{s + 1}} (k - s)^{\overline{s - r + 1}} }{ (v - 2 s + 1)^{\overline{s}} (v - 2 s + 1)^{\overline{s - r}} } \cdot \sum_{i = 0}^{r} (-1)^i {r \choose i} \frac{(k - r + 1)^{\overline{i}} (- v + s)^{\overline{i}}}{(-k)^{\overline{i}} (v - s - r + 1)^{\overline{i}}} & ( x^{\underline{n}} = (-1)^n x^{\overline{n}} ) \\
= & \frac{(k - s)^{\overline{s + 1}} (k - s)^{\overline{s - r + 1}} }{ (v - 2 s + 1)^{\overline{s}} (v - 2 s + 1)^{\overline{s - r}} }
\cdot \pFq{3}{2}{-r & k - r + 1 & - v + s}{ & -k & v - r - s + 1}{1} & (\text{by \cref{eq:outland}}) \\
= & \frac{ (k - s)^{\overline{s + 1}} (k - s)^{\overline{s - r + 1}} }{ (v - 2 s + 1)^{\overline{s}} (v - 2 s + 1)^{\overline{s - r}} }
\cdot \frac{r^{\underline{r / 2}} (v - k - s)^{\overline{r / 2}} }{ k^{\underline{r / 2}} (v - s - r + 1)^{\overline{r / 2}}} & (\text{by \cref{cor:Sunna}}) \\
= & \frac{r!}{(r / 2)!} G_{s, r}. & (\text{by \cref{eq:quadrantlike}})
\end{align*}
\end{proof}

\subsection{Denominators of $G_{s, r}$}
\label{sec:unflagitious}

Since $h_{s, i} \in \Q \otimes_{\Z} {\cal A}_s$ and $H_{s, r}$ is a rational quadratic polynomial in $h_{s, i}$, we have $H_{s, r} \in \Q \otimes_{\Z}$. Using the estimate of denominators of $h_{s, i}$'s in \cref{cor:epichorion}, we could estimate the denominator of $H_{s, r}$. Combined with \cref{thm:administration}, we could estimate the denominator of $G_{s, r}$.

\begin{proposition} \label{prop:foreadapt}
Let $s$ be a natural number, and $r$ an even natural number with $r \leq s$. Let
\begin{equation} \label{eq:retreative}
	F_{s, r} := s!^2 \ell_{s - 1}^2 \ell_{s - 2}^2 \frac{r!}{(r / 2)!} \in \Z.
\end{equation}
Then, $F_{s, r} G_{s, r} \in {\cal A}_s$.
\end{proposition}

\begin{proof}
We first rewrite $F_{s, r} G_{s, r}$ as follows.
\begin{align*}
F_{s, r} G_{s, r} & = \frac{(r / 2)!}{r!} F_{s, r} H_{s, r} & (\text{by \cref{thm:administration}}) \\
& = \sum_{i = 0}^r (-1)^i {r \choose i} \left(s! \ell_{s - 1} \ell_{s - 2} h_{s, s - r + i}\right) \left(s! \ell_{s - 1} \ell_{s - 2} h_{s, s - i}\right). & (\text{by \cref{eq:chafferer,eq:retreative}})
\end{align*}
Then, result follows from \cref{cor:epichorion}.
\end{proof}

\section{Upper bounds for $v$}
\label{sec:ineffaceability}

The goal of this section is to give upper bounds for $v$ in terms of $s$ and some other parameters that could be chosen.

Assume that $(v, k)$ are parameters of a nontrivial tight $2 s$-design. Recall that in \cref{prop:foreadapt} we proved that $F_{s, r} G_{s, r} \in {\cal A}_s$. Moreover, $F_{s, r} G_{s, r}$ is positive by \cref{eq:retreative,eq:quadrantlike}. In other words, $F_{s, r} G_{s, r}$ takes positive integer value at $(v, k)$. It is possible to obtain an upper bound for $v$ by analyzing the inequality $\ln F_{s, r} + \ln G_{s, r} \geq 0$. However, if we write down the complete arguments, which we will not do in this paper, it would only allow us to prove $v \ll s^{5}$ as $s \to \infty$.

We can use this idea to obtain a much better bound, as long as we apply a small twist. Let
\begin{equation}
\label{eq:camara}
	\tilde{G}_{s, r} := \frac{ {v - k - s + r / 2 - 1\choose r / 2} {k - r / 2 \choose s - r / 2 + 1} {k - r \choose s - r + 1}}{ {v - s \choose s} {v - s - r / 2 \choose s - r / 2} } = G_{s, r} \frac{s!}{(r / 2)! (s - r + 1)! (s - r / 2 + 1)} \in \Q(v, k).
\end{equation}
and
\begin{equation}
\label{eq:landscape}
	\tilde{F}_{s, r} := s! r! \ell_{s - 1}^2 \ell_{s - 2}^2 (s - r + 1)! (s - r / 2 + 1) = F_{s, r} \frac{(r / 2)! (s - r + 1)! (s - r / 2 + 1)}{s!} \in \Z.
\end{equation}
Clearly, $\tilde{F}_{s, r} \tilde{G}_{s, r} = F_{s, r} G_{s, r} \in {\cal A}_s$ by \cref{eq:camara,eq:landscape}.

The constant $\tilde{F}_{s, r}$ contains a lot of small primes, which are usually more than enough to make $\tilde{F}_{s, r} \tilde{G}_{s, r}$ an integer in $\Z_{(p)}$. This suggests us to choose carefully an $\tilde{F}^b_{s, r}$, which is much smaller than $\tilde{F}_{s, r}$, and prove that $\tilde{F}^b_{s, r} \tilde{G}_{s, r} \in {\cal A}_s$.

We first construct such an $\tilde{F}^b_{s, r}$ in \cref{sec:preselect}. In \cref{sec:appointee}, by examining the inequality $\ln \tilde{F}^b_{s, r} + \ln \tilde{G}_{s, r} \geq 0$, \cref{thm:persuasibility} gives an upper bound for $v$ in terms of $s$, $r$ and $b$. If we choose some special parameters $r$ and $b$ in \cref{thm:persuasibility}, then it gives $v \ll s$ as $s \to \infty$. In \cref{sec:babyfied}, we give an explicit version of this asymptotic result that $v \leq 2,000,000 s$ for $s \geq 627$.

\subsection{Construction of $\tilde{F}^b_{s, r}$}
\label{sec:preselect}

We construct $\tilde{F}^b_{s, r}$ by reducing the valuation of primes $p \leq b$ in $\tilde{F}_{s, r}$ artificially. Let
\begin{equation} \label{eq:platformism}
	\tilde{F}^b_{s, r} := \prod_{\text{prime $p \in [2, b]_{\Z}$}} p^{2 \lfloor \log_p (v - s) \rfloor} \prod_{\text{prime $p \in (b, s]_{\Z}$}} p^{{\rm val}_p(\tilde{F}_{s, r})}.
\end{equation}

\begin{proposition} \label{prop:imi}
Let $r$ be an even natural number with $r \leq s$. If there exists a nontrivial tight $2s$-$(v, k, \lambda)$ design, then $\tilde{F}^b_{s, r} \tilde{G}_{s, r}$ is a positive integer.
\end{proposition}

\begin{proof}
Since the design is nontrivial, $v - k \geq 2 s + 1$ and $k \geq 2 s + 1$. Then, the factorization of $G_{s, r}$ in \cref{eq:quadrantlike} shows that $G_{s, r}$ is a positive rational number, hence so does $\tilde{F}^b_{s, r} \tilde{G}_{s, r}$ by \cref{eq:camara,eq:platformism}.

For primes $p \leq b$,
\begin{align*}
	& {\rm val}_p(\tilde{F}^b_{s, r} \tilde{G}_{s, r}) \\
	\geq & 2 \lfloor \log_p (v - s) \rfloor - {\rm val}_p({v - s \choose s}) - {\rm val}_p({v - s - r / 2 \choose s - r / 2}) & (\text{by \cref{eq:camara,eq:platformism}}) \\
	\geq & 2 \lfloor \log_p (v - s) \rfloor - \lfloor \log_p (v - s) \rfloor - \lfloor \log_p (v - s - r / 2) \rfloor & (\text{since ${\rm val}_p({n \choose m}) \leq \lfloor \log_p n \rfloor$}) \\
	\geq & 0.
\end{align*}
For primes $p \gt b$,
\begin{align*}
	{\rm val}_p(\tilde{F}^b_{s, r} \tilde{G}_{s, r})
	= & {\rm val}_p(\tilde{F}_{s, r} \tilde{G}_{s, r}) & (\text{by \cref{eq:platformism}}) \\
	= & {\rm val}_p(F_{s, r} G_{s, r}) & (\text{by \cref{eq:camara,eq:landscape}}) \\
	\geq & 0. & (\text{by \cref{prop:foreadapt}})
\end{align*}
Therefore, $\tilde{F}^b_{s, r} \tilde{G}_{s, r}$ is a positive integer.
\end{proof}

\subsection{An upper bound for $v$}
\label{sec:appointee}

Since $\tilde{F}^b_{s, r} \tilde{G}_{s, r}$ is a positive integer by \cref{prop:imi}, it follows that $\ln \tilde{F}^b_{s, r} + \ln \tilde{G}_{s, r} \geq 0$. Analyzing this inequality gives an upper bound for $v$.

\begin{theorem} \label{thm:persuasibility}
Let $s$ be a natural number, $r$ an even natural number with $6 \leq r \leq s$, and $b$ a real number with $b \leq s$. Let
\begin{equation}
\label{eq:brideknot}
	\psi^b_{s, r} := r - 4 - 4 \pi(b),
\end{equation}
where $\pi(x)$ is the prime counting function that counts the number of primes no greater than $x$. Let
\begin{equation}
\label{eq:Keraunia}
\begin{aligned}
	\kappa^b_{s, r} := & 2 \left( \ln \frac{(2 s - 3 r / 2 + 2)! s!}{(2 s - r + 2)! (s - r + 1)! (s - r / 2 + 1)} + (s - r) \ln 2 \right) \\
	& + 2 \left( \sum_{\text{prime $p \in (b, s]_{\Z}$}} {\rm val}_p\left(\tilde{F}_{s, r}\right) \ln p \right) + (4 \ln \frac{3}{2}) \pi(b).
\end{aligned}
\end{equation}
Suppose that $\psi^b_{s, r} \gt 0$. If a nontrivial tight $2 s$-$(v, k, \lambda)$ design exists, then
\[
	v - 2 s + 1 \leq \exp(\kappa^b_{s, r} / \psi^b_{s, r}).
\]
\end{theorem}

\begin{proof}
By \cref{prop:imi}, $\tilde{F}^b_{s, r} \tilde{G}_{s, r}$ is a positive integer, hence
\begin{equation} \label{eq:sleeting}
	0 \leq 2 \ln \left( \tilde{F}^b_{s, r} \tilde{G}_{s, r} \right) = 2 \ln \frac{\tilde{G}_{s, r}}{G_{s, r}} + 2 \ln \tilde{F}^b_{s, r} + 2 \ln G_{s, r}.
\end{equation}

We use the properties of the falling factorials and the rising factorials to estimate $G_{s, r}$ as follows.
\begin{align*}
	G_{s, r}
= & \frac{ (v - k - s)^{\overline{r / 2}} (k - s)^{\overline{s - r / 2 + 1}} (k - s)^{\overline{s - r + 1}} }{ (v - 2 s + 1)^{\overline{s}} (v - 2 s + 1)^{\overline{s - r / 2}} } & (\text{by \cref{eq:quadrantlike}}) \\
\leq & \frac{ (v - k - s)^{\overline{r / 2}} (k - s)^{\overline{2 s - 3 r / 2 + 2}} }{ (v - 2 s + 1)^{\overline{s}} (v - 2 s + 1)^{\overline{s - r / 2}} } & \text{(since $x^{\overline{a}} x^{\overline{b}} \leq x^{\overline{a + b}}$)} \\
= & \frac{ (v - k - s + r / 2 - 1)^{\underline{r / 2}} (k + s - 3 r / 2 + 1)^{\underline{2 s - 3 r / 2 + 2}} }{ (v - s)^{\underline{s}} (v - 2 s + 1)^{\overline{s - r / 2}} } & \text{(since $x^{\overline{a}} = (x + a - 1)^{\underline{a}}$)} \\
\leq & \frac{1}{ {2 s - r + 2 \choose r / 2} } \frac{ (v - r)^{\underline{2 s - r + 2}} }{ (v - s)^{\underline{s}} (v - 2 s + 1)^{\overline{s - r / 2}} } & \text{(since ${a + b \choose a} x^{\underline{a}} y^{\underline{b}} \leq (x + y)^{\underline{a + b}}$)} \\
= & \frac{1}{ {2 s - r + 2 \choose r / 2} } \frac{ (v - 2 s)^{\underline{2}} (v - r)^{\underline{s - r}} }{ (v - 2 s + 1)^{\overline{s - r / 2}} } \\
\leq & \frac{1}{ {2 s - r + 2 \choose r / 2} } \frac{v^{s - r}}{ (v - 2 s + 1)^{s - r / 2 - 2}} & \text{(since $x^{\underline{a}} \leq x^a \leq x^{\overline{a}}$)} \\
\leq & \frac{2^{s - r}}{ {2 s - r + 2 \choose r / 2} } \frac{1}{ (v - 2 s + 1)^{r / 2 - 2}}. & \text{(since $v \leq 2 (v - 2 s - 1)$)}
\end{align*}
Taking logarithm on both sides gives the following estimate on $2 \ln G_{s, r}$.
\begin{equation} \label{eq:unslumbering}
	2 \ln G_{s, r} \leq - (r - 4) \ln (v - 2 s + 1) - 2 \ln {2 s - r + 2 \choose r / 2}  + 2 (s - r) \ln 2.
\end{equation}

Substituting \cref{eq:platformism,eq:unslumbering} into \cref{eq:sleeting},
\begin{align*}
	0
	\leq & 2 \ln \frac{\tilde{G}_{s, r}}{G_{s, r}} + 4 \pi(b) \ln (v - s) + 2 \left(\sum_{\text{prime $p \in (b, s]_{\Z}$}} {\rm val}_p\left(\tilde{F}_{s, r}\right) \ln p \right) \\
	& - (r - 4) \ln (v - 2 s + 1) - 2 \ln {2 s - r + 2 \choose r / 2}  + 2 (s - r) \ln 2 \\
	= & - (r - 4) \ln (v - 2 s + 1) + 4 \pi(b) (\ln (v - s) - \ln \frac{3}{2}) + \kappa^b_{s, r} & (\text{by \cref{eq:camara,eq:Keraunia}}) \\
	\leq & - (r - 4) \ln (v - 2 s + 1) + 4 \pi(b) \ln (v - 2 s + 1) + \kappa^b_{s, r} & (\text{since $v \geq 2 (2 s + 1)$ and $6 \leq r \leq s$}) \\
	= & - \psi^b_{s, r} \ln (v - 2 s + 1) + \kappa^b_{s, r}. & (\text{by \cref{eq:brideknot}})
\end{align*}
Since $\psi^b_{s, r} \gt 0$, we have $\ln (v - 2 s + 1) \leq \kappa^b_{s, r} / \psi^b_{s, r}$, from which the result follows.
\end{proof}

\subsection{Explicit upper bound for $v$ when $s \geq 627$}
\label{sec:babyfied}

From \cref{thm:persuasibility}, it is easy to get an asymptotic upper bound for $v$. Choose $b := s$ and $r := 2 \lfloor s / 2 \rfloor$. As $s \to \infty$, $\psi^b_{s, r} \asymp s$ and $\kappa^b_{s, r} \asymp s \ln s$. Therefore, $v \lt \exp(\kappa^b_{s, r} / \psi^b_{s, r}) + 2 s \ll s$ if there exists a nontrivial tight $2s$-$(v, k, \lambda)$ design. We provide an explicit version of this asymptotic bound in \cref{prop:premeditation}.

\begin{proposition}
\label{prop:premeditation}
Let $b := s$ and $r := 2 \lfloor s / 2 \rfloor$. If $s \geq 627$, then $\psi^b_{s, r} \gt 0$ and
\[
	\exp(\kappa^b_{s, r} / \psi^b_{s, r}) \lt 2,000,000 s.
\]
\end{proposition}

\begin{proof}
Dusart gave in \cite[Corollary 5.2]{Dusart2018} an explicit upper bound on the prime counting function,
\begin{equation} \label{eq:fortress}
	\pi(x) \leq \frac{x}{\ln x} \left(1 + \frac{1.2762}{\ln x}\right).
\end{equation}
With this, we estimate $\psi^b_{s, r}$.
\begin{align*}
\psi^b_{s, r}
\geq & s - 5 - 4 \pi(s) & (\text{by \cref{eq:brideknot} and $r \in \{s - 1, s\}$}) \\
\gt & s - 5 - 4 \frac{s}{\ln s} \left(1 + \frac{1.2762}{\ln s}\right) & (\text{by \cref{eq:fortress}}) \\
\gt & 0. & (\text{since $s \geq 172$})
\end{align*}

The Stirling's formula says that for $n \geq 1$,
\begin{equation} \label{eq:ornithoid}
	f(n) \lt \ln n! \leq f(n) + 1.
\end{equation}
where $f(n) := n \ln n - n + \frac{1}{2} \ln n$. With this and Dusart's bound, we estimate $\kappa^b_{s, r}$.
\begin{align*}
	\kappa^b_{s, r} \leq & 2 \ln (s / 2 + 7 / 2)! + 2 \ln s! - 2 \ln (s + 2)! \\
	& - 2 \ln (s / 2 + 1) + 4 (\ln \frac{3}{2}) \pi(s) & (\text{by \cref{eq:Keraunia} and $r \in \{s - 1, s\}$}) \\
	\leq & 2 f(s / 2 + 7 / 2) + 2 f(s) - 2 f(s + 2) + 4 \\
	& - 2 \ln (s / 2 + 1) + 4 (\ln \frac{3}{2}) \frac{s}{\ln s} \left(1 + \frac{1.2762}{\ln s}\right) & (\text{by \cref{eq:fortress,eq:ornithoid}}) \\
	\lt & \left(\ln s - \frac{5}{4}\right) s. & (\text{since $s \geq 186$})
\end{align*}

Therefore, combining the estimates of $\psi^b_{s, r}$ and $\kappa^b_{s, r}$,
\begin{align*}
& \kappa^b_{s, r} / \psi^s_{s, r} - \ln s \\
\leq & \left(\left(\ln s - \frac{5}{4}\right) s\right) \left(s - 5 - 4 \frac{s}{\ln s} \left(1 + \frac{1.2762}{\ln s}\right)\right)^{-1} - \ln s \\
\lt & \ln 2,000,000, & (\text{since $s \geq 627$})
\end{align*}
from which the result follows.
\end{proof}

\section{Lower bounds for $v$}
\label{sec:gaffer}

For a nontrivial tight $2s$-$(v, k, \lambda)$ design, the parameter $\lambda$ is just $\lambda_{s, s}$ in \cref{eq:Bee}. Moreover,
\begin{equation}
\label{eq:impersuasibly}
	s! \lambda_{s, s} = \frac{k \cdots (k - 2 s + 1)}{(v - s) \cdots (v - 2 s + 1)}
\end{equation}
takes integral value at parameters of nontrivial tight designs. This means that the denominator of the right side should not contain any prime factor larger than $k$. It indicates that we can use number theoretic properties to give a lower bound for $v$.

\begin{theorem}
\label{thm:cricoid}
Let $\rho_s$ be the smallest positive integer $n$ such that there are no primes in $(n, n + s - 1]_{\Z}$. If there exists a nontrivial tight $2 s$-$(v, k, \lambda)$ design, then
\[
	v - 2 s \geq \rho_{s + 1}.
\]
\end{theorem}

\begin{proof}
As shown in \cref{eq:impersuasibly}, the numerator of $s! \lambda_{s, s}$ is the product of $2 s$ consecutive integers from $k - 2 s + 1$ to $k$, and the denominator of $s! \lambda$ is the product of $s$ consecutive integers from $v - 2 s + 1$ to $v - s$. The design is nontrivial, hence $v - 2 s + 1 \gt k$. Since $\lambda_{s, s}$ is an integer, neither of the $s$ consecutive integers from $v - 2 s + 1$ to $v - s$ could be a prime. Therefore, $v - 2 s \geq \rho_{s + 1}$.
\end{proof}

\begin{remark}
\label{rem:benzophenanthroline}
The number $\rho_s$ can be rewritten as
\[
	\rho_s = \min \{p_n \mid p_{n + 1} - p_n \geq s\},
\]
where $p_n$ is the $n$-th prime. Due to its connection with prime gap, the exact values of $\rho_s$ are known for $s \leq 1550$, which is the largest known maximal prime gap. In particular, $\rho_{288} = 1,294,268,491$.
\end{remark}

An upper bound on prime gap will give us a lower bound of $\rho_s$. We use one upper bound given by Dusart. Dusart's result is not asymptotically the best known one, but it is explicit and has very nice constants.

\begin{proposition} \label{prop:pithecomorphism}
For $s \geq 288$,
\[
	\rho_s \gt 5000 s (14.6 + \ln s)^2.
\]
In particular,
\[
	\rho_s \gt 2,000,000 s.
\]
\end{proposition}

\begin{proof}
Dusart showed in \cite[Corollary 5.5]{Dusart2018} that for $x \geq 468,991,632$, there exists a prime $p$ in the interval
\[
	\left(x, x + \frac{x}{5000 (\ln x)^2}\right].
\]

Since $\rho_s \geq \rho_{288} \gt 468,991,632$ by \cref{rem:benzophenanthroline}, $\left(\rho_s, \rho_s + \frac{\rho_s}{5000 (\ln \rho_s)^2}\right]$ contains a prime. On the other hand, by definition, $\left(\rho_s, \rho_s + s - 1\right]$ contains no primes. Therefore,
\[
	\frac{\rho_s}{5000 (\ln \rho_s)^2} \geq s.
\]
Then,
\begin{align*}
	\rho_s \geq & 5000 s (\ln \rho_s)^2 \geq 5000 s (\ln (5000 s (\ln \rho_s)^2))^2 \geq 5000 s (\ln (5000 s (\ln \rho_{288})^2))^2 \\
	\gt & 5000 s (14.6 + \ln s)^2 \gt 2,000,000 s. \qedhere
\end{align*}
\end{proof}

\begin{remark} \label{rem:osteomere}
The result by Baker-Harman-Pintz \cite{BakerHarmanPintz2001} on prime gaps in 2001 leads to $\rho_s \gg s^{40 / 21}$. A conditional result by Cram\'er \cite{Cramer1936} in 1936 leads to $\rho_s \gg s^2 / (\ln s)^2$ assuming Riemann hypothesis. Cram\'er's conjecture \cite{Cramer1936} suggests that we may have $\ln \rho_s \gg \sqrt{s}$.
\end{remark}

\section{Proof of the main result}
\label{sec:malandrous}

\cref{sec:ineffaceability} gives upper bounds for $v$ and \cref{sec:gaffer} gives lower bounds for $v$. Now, we have all the ingredients needed to prove the main result \cref{thm:myxocyte}.

\begin{proof}[Proof of \cref{thm:myxocyte}]
Through a series of previous works \cite{Bremner1979,Peterson1977,Xiang2018,DukesShortGershman2013}, it has been proved that the only nontrivial tight $2s$-designs with $s \in [2, 9]_{\Z}$ are the Witt $4$-$(23, 7, 1)$ design and its complement the Witt $4$-$(23, 16, 52)$ design. From now on we assume that $s \geq 10$, and that there exists a nontrivial tight $2 s$-$(v, k, \lambda)$ design.

Since the design is nontrivial, $k \geq 2 s + 1$ and $v - k \geq 2 s + 1$. Its complementary design is a nontrivial tight $2 s$-$(v, v - k, \overline{\lambda})$ design for some parameter $\overline{\lambda}$. Since the complementary design has the same $s$ and $v$ as the original $2 s$-$(v, k, \lambda)$ design, we may assume without loss of generality that $v \leq 2 k$.

Choose $r := 2 \lfloor s / 2 \rfloor$. Combining \cref{thm:persuasibility} and \cref{thm:cricoid},
\begin{equation} \label{eq:electrothermometer}
	\rho_{s + 1} \leq v - 2 s \lt \exp(\kappa^b_{s, r} / \psi^b_{s, r}),
\end{equation}
whenever $\psi^b_{s, r} \gt 0$ for some $b \leq s$.

We divide the situation into three cases depending on the size of $s$, and show contradictions in each case.

\smallskip

\noindent {\bf Case 1}: $s \geq 627$.

Choose $b := s$. \cref{prop:premeditation} shows $\psi^b_{s, r} \gt 0$ and gives an upper bound $\exp(\kappa^b_{s, r} / \psi^b_{s, r}) \lt 2,000,000 s$. On the other hand, \cref{prop:pithecomorphism} gives a lower bound $\rho_s \gt 2,000,000 s$. These two bounds contradict with \cref{eq:electrothermometer}.

\smallskip

\noindent {\bf Case 2}: $s \in [288, 626]_{\Z}$.

Choose $b := s / 3$. Direct calculations show that $\psi^b_{s, r} \gt 0$ and $\exp(\kappa^b_{s, r} / \psi^b_{s, r}) \lt 1,000,000,000$. On the other hand, $\rho_{s + 1} \geq \rho_{288} \gt 1,000,000,000$ by \cref{rem:benzophenanthroline}. These two bounds contradict with \cref{eq:electrothermometer}.

\smallskip

\noindent {\bf Case 3}: $s \in [10, 287]_{\Z}$.

For each $s \in [10, 287]_{\Z}$, calculations over $b \in [1, s]_{\Z}$ shows that we can choose a suitable $b$ such that $\psi^b_{s, r} \gt 0$ and $\exp(\kappa^b_{s, r} / \psi^b_{s, r}) \lt 15,000,000,000$. Therefore, by \cref{eq:electrothermometer}, $v - 2 s \leq 15,000,000,000$.

Let $x := k - s$ and $y := v - 2 s + 1$. Then, $\alpha_{s, i}$ in \cref{eq:Bass} can be rewritten as
\begin{equation} \label{eq:impairment}
\alpha_{s, i} = {s \choose i} \frac{x^{\overline{i}}(x + 1)^{\overline{i}}}{y^{\overline{i}}}.
\end{equation}
The $\alpha_{s, i}$'s are integers by \cite[Corollary 2.2]{Xiang2018}. Moreover, $y \in [x + s + 2, 2 x + 1]_{\Z}$ and $x \lt k \leq v - 2 s - 1 \leq 15,000,000,000$.

However, computer search\footnote{The source code and the output of the program used can be found on the author's website: \href{https://ziqing.org/tight-block-design}{https://ziqing.org/tight-block-design}.} shows that there are no integer triples $(s, x, y)$ with $s \in [10, 287]_{\Z}$, $x \in [1, 15,000,000,000]_{\Z}$ and $y \in [x + s + 2, 2 x + 1]_{\Z}$ such that $\alpha_{s, 1}, \dots, \alpha_{s, 6}$ in \cref{eq:impairment} are all integers.
\end{proof}

\printbibliography

\end{document}